\documentclass[12pt]{amsart}
\usepackage[latin1,utf8]{inputenc}
\usepackage{fancyhdr}
\usepackage{graphicx}
\usepackage{newlfont}
\usepackage{amssymb}
\usepackage{amsmath}
\usepackage{latexsym}
\usepackage{fullpage}
\usepackage{amsthm}
\usepackage{xcolor}
\usepackage{enumitem}
\usepackage{comment}
\usepackage{tikz}
\usepackage{tikz-cd}
\usepackage[T1]{fontenc}
\usepackage{array} 
\usepackage{makecell}
\usepackage{hyperref}
\definecolor{coolblack}{rgb}{0.0, 0.18, 0.39}
\definecolor{indiagreen}{rgb}{0.07, 0.53, 0.03}
\hypersetup{ colorlinks=true, linkcolor=coolblack, citecolor=indiagreen, urlcolor=blue}
\usepackage{cleveref}

\newcommand{\barb}{\bar B}
\newcommand{\bara}{\bar A}

\theoremstyle{plain}
\newtheorem*{thm*}{Theorem}
\newtheorem*{con*}{Conjecture}
\newtheorem{thm}{Theorem}[section]

\newtheorem{lem}[thm]{Lemma}
\newtheorem{con}[thm]{Conjecture}

\newtheorem{cor}[thm]{Corollary}

\newtheorem*{str*}{Strategy}

\theoremstyle{definition}
\newtheorem{defn}[thm]{Definition}

\theoremstyle{remark}

\usepackage{enumitem}

\title{Double shortcuts of standard hypercube decompositions}

\author{Margherita Zannoni}

\date{}

\address{Margherita Zannoni, Dipartimento di Ingegneria e Scienze dell'Informazione e Matematica, Universit\`a degli Studi dell'Aquila, Via Vetoio SNC, 67100 L'Aquila, Italy}

\begin{document}

\begin{abstract}
In this paper, we study the double shortcuts associated with pairs of standard hypercube decompositions of arbitrary Bruhat intervals in the symmetric group. Our results imply that a conjecture stated in [Bull. London Math. Soc., 57 (2025), no. 8] holds for the class of standard hypercube decompositions. If this conjecture were to hold for all hypercube decompositions, then the Combinatorial Invariance Conjecture for Kazhdan--Lusztig polynomials would follow.
\end{abstract}

\maketitle


\section{Introduction}

The Combinatorial Invariance Conjecture, independently formulated by Lusztig and Dyer in the 1980s, is one of the most intriguing open problems in the theory of Kazhdan-Lusztig polynomials. This conjecture asserts that Kazhdan-Lusztig polynomials depend only on the isomorphism type of the Bruhat interval as a poset.

\begin{con}[Combinatorial Invariance]
\label{cic}
Let $W, W'$ be Coxeter groups, $u, v \in W$, $u', v' \in W'$. If $[u, v] \cong [u', v']$ as posets, then $P_{u,v} (q) = P_{u',v'} (q)$ (or, equivalently, $\widetilde R_{u,v} (q) =\widetilde R_{u',v'} (q)$).
\end{con}

The conjecture has been focus of active research for the past forty years and it is known to hold in various partial cases. For instance, it has been proved for intervals of
length $\leq 6$ (\cite{BGL}), for intervals in type $\tilde A_2$ (\cite{BLP}), for intervals that are lattices (\cite{Dyeth}, 7.23, and \cite{Bre94}), and for intervals that start from the identity of the group (\cite{BCM1}). A parabolic version of  Conjecture \ref{cic} holds for intervals starting from the identity (\cite{Mtrans}, \cite{M}), while certain purely poset-theoretic generalizations  of the conjecture fail (\cite{BCM2}, \cite{MJaco}).

In recent years, two new approaches have been proposed to investigate this conjecture: via  hypercube decompositions and via flipclasses. The latter approach (\cite{EM}, \cite{EMS}) has led to the proof of the combinatorial invariance of  Kazhdan--Lusztig $\widetilde{R}$-polynomials of  Weyl groups modulo $q^7$ and of  Kazhdan--Lusztig $\widetilde{R}$-polynomials of type $A$  Weyl groups modulo $q^8$. As a consequence, the Combinatorial Invariance Conjecture holds for all intervals up to length 8 in Weyl groups and up to length 10 in type $A$.

The approach via hypercube decomposition, introduced in \cite{BBDVW} and \cite{DVBBZTTBBJLWHK} and developed  through the use of certain machine learning methods, appears  very promising for tackling the combinatorial invariance problem in the symmetric group.
This technique has led to conjectural recursive description of the $P$-polynomials (\cite{BBDVW}) and of the $\widetilde R$-polynomials (\cite{BM}, and \cite{BG}). The analisys in  \cite{BG} establishes  the conjecture for elementary intervals in symmetric groups.

The description in \cite{BM} is formulated in terms of shortcuts of a hypercube decomposition. A recent work (\cite{EM2}) introduces the notions of amazing hypercube decompositions and, given two amazing hypercube decompositions $z$ and $z'$, that of $(z,z')$-double shortcut. This has led  to a new conjecture implying the Combinatorial Invariance Conjecture for type~$A$. 
Let $DS(z,z')$ denote the multiset of  $(z,z')$-double shortcuts. 

\begin{con}\label{EMtrivial1}
Let $W$ be a Coxeter group of type $A$. The  equivalence relation generated by $z \sim z'$ if $DS(z,z')= DS(z',z)$ is trivial (that is, it has only one equivalence class).
\end{con} 
The stronger symmetry condition that $DS(z,z')= DS(z',z)$ holds for every pair of amazing hypercube decompositions has been verified through computer calculations  up to $A_5$ (see \cite{EM2}).
In the appendix of \cite{EM2}, Barkley and Gaetz prove that Conjecture \ref{EMtrivial1} holds for all co-elementary
 intervals. 

In this paper, we show that all standard hypercube decompositions of a given interval belong to the same equivalence class (Corollary~\ref{cor}). In particular,  Conjecture \ref{EMtrivial1}  holds for  all intervals having no hypercube decompositions apart from the standard ones.

\medskip

Our result provides structural evidence in support of the double-shortcut framework. 



\section{Preliminaries}

We begin by setting the notation and recalling basic definitions. For standard background on Coxeter groups, we refer the reader to \cite{BB}. 

The notion of hypercube decomposition is relevant only in the case of Bruhat intervals of type $A$. So, from now on we let $W$ be the Coxeter group of type $A_{n-1}$, that is, the symmetric group $\mathfrak{S}_n$ with simple transpositions as generators.

We recall the definition of hypercube decomposition of a Bruhat interval of $W$ following \cite{BM}, dually from how it was originally introduced in \cite{BBDVW}. Let $E$ be a finite set. We call \emph{hypercube} the directed acyclic graph $H_E$ whose vertices are the subsets of $E$ and whose edges connect two elements if and only if the second is obtained by adding one element to the first one.
Let $p\in W$. Let $E$ be a set of edges of $\mathcal{B}(W)$ having the same target $p$. We say that $E$ \emph{spans a hypercube} if there exists a unique embedding of directed graphs $\theta_E : H_E \rightarrow \mathcal{B}(W)$ such that, for any $\alpha\in E$, the edge ($E\setminus\left\{\alpha\right\}\to E$) is sent to $\alpha\in \mathcal{B}(W)$. Eventually, we say that $E$ \textit{spans a hypercube cluster} if every subset $E'$ of $ E$ consisting of edges with pairwise incomparable sources spans a hypercube. 
\begin{defn}
Let $u,v\in W$ and $z\in[u,v]$. We say that $z$ is a \textit{hypercube decomposition} of the interval $[u,v]$ provided that:
\begin{enumerate}[label=\roman*.]
\item $[z,v]$ is \emph{diamond complete} with respect to $[u,v]$, meaning that, if there exist $x\in [u,v]$ and $a_1,a_2,y \in [z,v]$, $a_1 \neq a_2$, such that $x\rightarrow a_1\rightarrow y$, $x \rightarrow a_2 \rightarrow y$, then $x\in[z,v]$; 
\item for all $p\in [z,v]$, the set $E^p=\left\{x\to p:x\in X\setminus [z,v]\right\}$ spans a hypercube cluster.
\end{enumerate}
\end{defn}

If $\Gamma = (u = x_0 \to x_1 \to \cdots \to x_r = v)$ is a Bruhat path from $u$ to $v$ in $\mathcal{B}(W)$, then $r$ is referred to as the length of $\Gamma$, denoted $\ell(\Gamma)$, and $\{x_i: i \in [0, r]\}$ is called the support of $\Gamma$, denoted supp$(\Gamma)$. Furthermore, the distance from $u$ to $v$, denoted $d(u, v)$, is $\min\{\ell(\Gamma) : \Gamma \text{ is a path from }u\text{ to }v\}$.

Following \cite{BM}, we give the definitions of shortcuts and $R$-elements.

\begin{defn} 
Let $u,v\in W$ and $z\in[u,v]$.
    The set of \textit{shortcuts} for $[u,v]$ with respect to $z$ is defined as $$W^z_{[u,v]}= \{   p\in [z,v] : \text{supp}(\Gamma)\cap [z,v]=\left\{p\right\} \text{ for all paths } \Gamma\text{ from }u\text{ to } p \text{ of length } d(u,p)  \}.$$
    Moreover, we say that $z$ is an $R$-element for $[u,v]$ if \begin{equation*}\widetilde{R}_{u,v}(q)=\sum_{p\in W^z_{[u,v]}} q^{d(u,p)}\widetilde{R}_{p,v}(q).
\end{equation*}
\end{defn}
Fix $u,v\in W$, with $u\leq v$. The elements $$z_n=\min([u,v]\cap W_{S\setminus \{s_{n-1}\}}v) ,\hspace{5mm} z_1=\min([u,v]\cap W_{S\setminus \{s_{1}\}} v),$$ 
$$z^n=\min([u,v]\cap  v W_{S\setminus \{s_{n-1}\}} ), \hspace{5mm} z^1=\min([u,v]\cap v  W_{S\setminus \{s_{1}\}} )$$ are always  hypercube decompositions (possibly  not mutually distinct) of the interval $[u,v]$. We call these elements the \emph{standard} hypercube decompositions of $[u,v]$. It has been shown that standard hypercube decompositions are $R$-elements (\cite{BM}, Theorem~3.7 and Corollary~3.10).

The following theorem provides us with a combinatorial characterization of the set of shortcuts for a Bruhat interval with respect to standard hypercube decompositions (\cite{BM}, Corollary~4.3 and Remark~4.4). 

\begin{thm}\label{char}Let $u,v\in W$, with $u\leq v$.
\begin{enumerate}[label=\arabic*.]
\item The shortcuts for $[u,v]$ with respect to $z_n$ are precisely the elements of the form $(n,a_k,\ldots,a_2,a_1)u$ for some $1\leq a_1<a_2<\ldots <a_k<n$ such that $v^{-1}(n)=u^{-1}(a_1)<u^{-1}(a_2)<\ldots <u^{-1}(a_k)<u^{-1}(n)$.
    \item The shortcuts for $[u,v]$ with respect to $z_1$ are precisely the elements of the form $(1,a_1,\ldots,a_{k-1},a_k)u$ for some $1<a_1<\ldots<a_{k-1}<a_k\leq n$ such that $u^{-1}(1)<u^{-1}(a_1)<\ldots<u^{-1}(a_{k-1})<u^{-1}(a_k)=v^{-1}(1)$.
    \item The shortcuts for $[u,v]$ with respect to $z^n$ are precisely the elements of the form $(\alpha,a_1,\ldots,a_{k-1},a_k)u$ for some $\alpha<a_1<\ldots<a_{k-1}<a_k\leq n$ such that $u^{-1}(\alpha)<u^{-1}(a_1)<\ldots<u^{-1}(a_{k-1})<u^{-1}(a_k)=n$, where $\alpha=v(n)$.
    \item The shortcuts for $[u,v]$ with respect to $z^1$ are precisely the elements of the form $(\beta,a_k,\ldots,a_2,a_1)u$ for some $1\leq a_1<a_2<\ldots<a_k<\beta$ such that $1=u^{-1}(a_1)<u^{-1}(a_2)<\ldots<u^{-1}(a_k)<u^{-1}(\beta)$, where $\beta=v(1)$.
\end{enumerate}
\end{thm}

The following definition, introduced in \cite{EM2}, strengthens the definition of \emph{join} hypercube decomposition originally presented in \cite{BM}. This is a combinatorially defined class of hypercube decompositions. 

\begin{defn}
   Let $u,v \in W$. We say that a  hypercube decomposition $z$ of $[u,v]$ is \textit{amazing} if the intersection $[z,v]\cap [x,v]$ has a minimum $z\vee x$ for all $x \in [u,v]$, and this minimum $z \vee x$ is a (necessarily  amazing) hypercube decomposition of $[x,v]$. Moreover, we say that $z$ is an \textit{amazing $R$-element} for $[u,v]$ if $z \vee x$ is an $R$-element for $[x,v]$ for all $x \in [u,v]$.
\end{defn}

Standard hypercube decompositions are amazing hypercube decompositions are amazing $R$-elements (see \cite{EM2}, Remark~3.2)

Whereas shortcuts relate to a single element in a Bruhat interval, double shortcuts, introduced in \cite{EM2}, are associated with an ordered pair of amazing hypercube decompositions.


\begin{defn}
    Let $u,v \in W$. Let $z$ and $z'$ be amazing hypercube decompositions of $[u,v]$. We say that an element $b$ in $[z,v]\cap [z',v]$ is a $(z,z')$-\emph{double shortcut} for $[u,v]$ if there exists $p\in W^z_{[u,v]}$ such that $b\in W^{z'\vee p}_{[p,v]}$. Denote by $DS(z,z')$ the  multiset 
$$\left\{(d(u,p) + d(p,b) , b ) : p\in W^z_{[u,v]} \text{ and } b\in W^{ z'\vee p}_{[p,v]}  \right\}.$$
\end{defn}

Fix $u,v \in W$. Following \cite{EM2}, we consider an equivalence relation on the set of amazing hypercube decompositions of $[u,v]$ defined as the transitive closure of the relation $z \sim z'$ if $DS(z,z')= DS(z',z)$. This led to the formulation of the following two conjectures (\cite{EM2}, Conjecture~3.6 and Conjecture~3.7). Conjecture~\ref{EMoneeqclass} is a weakening of Conjecture~\ref{EMtrivial}, but both conjectures imply the Combinatorial Invariance Conjecture.

\begin{con}\label{EMtrivial}
Let $W$ be a Coxeter group of type $A$. The  above relation is trivial (that is, it has only one equivalence class).
\end{con} 

\begin{con}\label{EMoneeqclass}
Let $W$ be a Coxeter group of type $A$. 
Every equivalence class of the above relation contains an amazing $R$-element.
\end{con}

Up to $A_5$, computer calculations have verified that Conjecture~\ref{EMtrivial} holds. Indeed, a much stronger statement has been
verified, that is, $$DS(z,z')=DS(z',z)$$ for all amazing hypercube decompositions of any interval in Coxeter groups of type $A_{n-1}$, for $n\leq 6$.

Our result supports Conjecture~\ref{EMtrivial}, proving that, for any interval in Coxeter groups of type $A_{n-1}$, all standard hypercube decompositions lie in the same equivalence class.


\section{Double shortcuts of standard hypercube decompositions}

From now on, $W$ always denotes the Coxeter group of type $A_{n-1}$ and $u,v\in W$ are fixed. 

Theorem~\ref{thm11} establishes that the multisets of double shortcuts associated with some pairs of standard hypercube decompositions are symmetric. Its proof reduces to a combinatorial statement, which we state separately as Lemma~\ref{lemma biezione1}.

We start by considering the pair $z_n$ and $z_1$. By Theorem~\ref{char}, for $p,b \in W$, we have that  $p\in W^{z_n}_{[u,v]}$ and $b\in W^{z_1\vee p}_{[p,v]}$ if and only if 
$p=A\cdot u$ and $b=B\cdot p\leq v$, where $A=(n,a_k, a_{k-1}, \ldots, a_1)$ and $B=(1,b_1,b_{2}\ldots,b_h)$ are cycles satisfying:
\begin{enumerate}[label=(\arabic*)]
\item \label{shzw1}
$1\leq a_1<a_2<\ldots <a_k<n$,
\item \label{shzw2}
$v^{-1}(n)=u^{-1}(a_1)<u^{-1}(a_2)<\ldots <u^{-1}(a_k)<u^{-1}(n)$,
\item\label{bshzw1} $1<b_1<\ldots<b_{h-1}<b_h < n$,
\item\label{bshzw2} $u^{-1}\cdot A^{-1}(1)<u^{-1}\cdot A^{-1}(b_1)<\ldots<u^{-1}\cdot A^{-1}(b_{h-1})<u^{-1}\cdot A^{-1}(b_h)=v^{-1}(1)$. 
\end{enumerate}
Notice that $b_h \neq n$, since $b\in[p,v]$ implies 
$u^{-1}\cdot B^{-1}\cdot A^{-1}(n)=u^{-1}\cdot A^{-1}(n)=v^{-1}(n)$.
On the other hand,  $a_1$ might be $1$. 

Symmetrically, given $q,b \in W$, we have that  $q\in W^{z_1}_{[u,v]}$ and $b\in W^{z_n\vee q}_{[q,v]}$ if and only 
$q=\bar{B}\cdot u$ and $b=\bar{A}\cdot q$, where   $\bar{B}=(1,\bar b_1,\bar b_{2}\ldots,\bar b_{\bar h})$ and $\bar{A}=(n,\bar a_{\bar k},\bar a_{\bar k-1}, \ldots,\bar a_1)$ are cycles satisfying:
\begin{enumerate}[label=(\alph*)]
\item \label{mshzw1}
$1<\bar b_1<\ldots<\bar b_{\bar{h}-1}<\bar b_{\bar h}\leq n$,
\item \label{mshzw2}
$u^{-1}(1)<u^{-1}(\bar b_1)<\ldots<u^{-1}(\bar b_{{\bar h}-1})<u^{-1}(\bar b_{\bar h})=v^{-1}(1)$. 
\item \label{mbshzw1}
$1\leq \bar a_1<\bar a_2<\ldots <\bar a_{\bar k}<n$,
\item \label{mbshzw2}
$v^{-1}(n)=u^{-1}\cdot \barb^{-1}(\bar a_1)<u^{-1}\cdot \barb^{-1}(\bar a_2)<\ldots <u^{-1}\cdot \barb^{-1}(\bar a_{\bar k})<u^{-1}\cdot \barb^{-1}(n)$; 
\end{enumerate}
Similarly, $\bar a_1\neq 1 $, since $b\in[q,v]$ implies $u^{-1}\cdot \bara^{-1}\cdot\barb^{-1}(1)=u^{-1}\cdot \barb^{-1}(1)=v^{-1}(1)$, but $\bar b_{\bar h}$ might be $n$.

Set 
\[
\mathcal{D} := \left\{ (A,B) : 
\parbox{7.8cm}{
    $A$ and $B$ are cycles $(n, a_k, a_{k-1},\ldots, a_1)$ and $(1, b_1, b_2,\ldots, b_h)$ 
    satisfying \ref{shzw1}, \ref{shzw2}, \ref{bshzw1}, \ref{bshzw2}
} \right\},
\]
and
\[
\mathcal{\bar D} := \left\{ (\barb,\bara) : 
\parbox{8.1cm}{
     $\bar{B}$ and $\bar{A}$ are cycles  $(1,\bar b_1,\bar b_{2},\ldots,\bar b_{\bar h})$ and $(n,\bar a_{\bar k}, \bar a_{\bar k-1}, \ldots,\bar a_1)$ satisfying \ref{mshzw1}, \ref{mshzw2},  \ref{mbshzw1}, \ref{mbshzw2}}
     \right\}.
\]

Given a cycle $C$, we denote by $C_s$ the set of elements moved by $C$.

\begin{lem}\label{lemma biezione1}
There exists a bijection \begin{align*}
    \phi: \hspace{2mm}& \mathcal{D} \rightarrow \mathcal{\bar{D}} \\ &(A,B)\mapsto(\bara,\barb),
\end{align*} such that $B \cdot A=\bara \cdot \barb$ and  $|A_s|+|B_s|=|\barb_s|+|\bara_s|$  for all $(A,B)\in\mathcal{D}$.
\end{lem}
\begin{proof}
Let $(A,B) \in \mathcal{D}$. We define $\phi(A,B)=(\bar{B},\bar{A})$ case-by-case. For each case, the following conditions hold:
\begin{enumerate}[label=(\roman*)]
\item
\label{m1}
$\bar{B}$ and $\bar{A}$ are cycles   satisfying \ref{mshzw1}, \ref{mshzw2},  \ref{mbshzw1}, \ref{mbshzw2};
\item \label{m2}
$B \cdot A=\bara \cdot \barb$;
\item \label{m3}
$|A_s|+|B_s|=|\bar{B}_s|+|\bar{A}_s|$.
\end{enumerate}
After that, we will show that $\phi$ is invertible and that its inverse satisfies symmetric properties. 
Let $A=(n,a_k, a_{k-1}, \ldots, a_1)$, $B=(1,b_1,b_{2}\ldots,b_h)$, and denote $a_0:=n$ and $b_{0}:=1$.

\underline{Definition of $\phi$:}

If $|A_s\cap B_s|=0$, we let $\bar{B}=B$ and $\bar{A}=A$. Since $A$ and $B$ commute, conditions \ref{m1}, \ref{m2}, and \ref{m3} are immediate.

If $|A_s\cap B_s|=m$, for some $m>0$, we let $\{c_1,c_2,\ldots,c_m\}$ be the elements of $A_s\cap B_s$ indexed in non-decreasing order. Then  $c_l=a_{i_l}=b_{j_l}$ for some $i_l\in \{1,\ldots,k\}$, $j_l\in\{0,1,\ldots,h\}$ such that $i_1<i_2<\ldots<i_m$ and $j_1<j_2<\ldots<j_m$. 
\newline
We start by assuming that \begin{equation*}
\label{distanza}\tag{$\ast$}
i_{r+1}-i_r=1 \hspace{3mm}\text{ and }\hspace{3mm}
j_{r+1}-j_r= 1 \hspace{3mm} \text{ for all } \hspace{3mm} r\in\{1,\ldots,m-1\}.
\end{equation*}
Notice that, in the one-line notation of $u$, the elements $a_1,a_2,\ldots,a_k$ must be in this order. The elements $b_1,b_2,\ldots,b_h$ must appear in order as well, except for $b_{j_1}$. In fact, it may happen that $b_{j_1}$ is to the left of $b_{j_1-1}$, since $b_{j_1}\in A_s$ and therefore it will be moved to the right when cycle $A$ is applied to $u$. The same cannot happen for $b_{j_l}$, for $l\in\{2,\ldots,m\}$, even though $b_{j_l}\in A_s$, because $b_{j_l}=a_{i_l}$.
\newline
Hence, we distinguish two main cases depending on whether $b_{j_1}$ lies to the right or to the left of $b_{j_1-1}$ in the one-line notation of $u$.  Overall, there are four cases, as each main case splits into two sub-cases depending on whether $a_{i_m+1}$ is greater than $b_{j_m+1}$. If $j_m=h$, then $b_{j_m+1}$ is undefined; in this instance, we adopt the convention to be in the case $a_{i_m+1}<b_{j_m+1}$ and then let $b_{j_m+1}:=1$. If $i_m=k$, we just let $a_{k+1}:=n$. 
\begin{itemize}
        \item If $u^{-1}(b_{j_1})>u^{-1}(b_{j_1-1})$, then the one-line notation of $u$ is of the form
        $$u=[ \hspace{0.7mm}\cdots \hspace{0.7mm}b_{j_1-1}\hspace{0.7mm}\cdots \hspace{0.7mm}\hspace{0.7mm}\substack{a_{ i_1}\\  \parallel \\b_{j_1}}\hspace{0.7mm}\cdots\hspace{0.7mm}\substack{a_{ i_2}\\  \parallel \\b_{j_2}}\hspace{0.7mm}\cdots \cdots\hspace{0.7mm} \substack{a_{ i_m}\\  \parallel \\b_{j_m}}\hspace{0.7mm}\cdots \hspace{0.7mm} a_{i_m+1}\hspace{0.7mm}\cdots\hspace{0.7mm}].$$
        \begin{enumerate}
            \item\label{caso1m} If $a_{i_m+1}>b_{j_m+1}$, let 
            $$\barb=B\hspace{5mm} \text{and} \hspace{5mm}\bara=(a_{i_m},b_{j_m+1})\cdot(a_{i_1},a_{i_{1}-1})\cdot A.$$
            This means that the cycle $\bara$ is equal to $A$ but with $a_{i_1}$ removed and $b_{j_m+1}$ added between $a_{i_m+1}$ and $a_{i_m}$.
            \item\label{caso2m} If $a_{i_m+1}<b_{j_m+1}$, let $$\barb=(b_{j_m+1},a_{i_m+1})\cdot B\hspace{5mm} \text{and} \hspace{5mm}\bara=(a_{i_1},a_{i_{1}-1})\cdot A.$$
            In this case, the cycle $\barb$ is equal to $B$ but with the element $a_{i_m+1}$ added between $b_{j_m}$ and $b_{j_m+1}$ and $\bara$ is obtained from $A$ by removing the element $a_{i_1}$.
    \end{enumerate}
    \item If $u^{-1}(b_{j_1})<u^{-1}(b_{j_1-1})$, then the one-line notation of $u$ will be of the form
        $$u=[ \hspace{0.7mm}\cdots \hspace{0.7mm}\substack{a_{ i_1}\\  \parallel \\b_{j_1}}\hspace{0.7mm}\cdots \hspace{0.7mm}\hspace{0.7mm}b_{j_1-1}\hspace{0.7mm}\cdots\hspace{0.7mm}\substack{a_{ i_2}\\  \parallel \\b_{j_2}}\hspace{0.7mm} \cdots\hspace{0.7mm}\substack{a_{ i_3}\\  \parallel \\b_{j_3}}\hspace{0.7mm} \cdots \cdots\hspace{0.7mm} \substack{a_{ i_m}\\  \parallel \\b_{j_m}}\hspace{0.7mm}\cdots\hspace{0.7mm} a_{i_m+1}\hspace{0.7mm}\cdots\hspace{0.7mm}].$$        
        \begin{enumerate}[resume]
            \item\label{caso3m} If $a_{i_m+1}>b_{j_m+1}$, let $$\barb=(b_{j_1},b_{j_2})\cdot B \hspace{5mm}\text{and} \hspace{5mm}\bara=(a_{i_m},b_{j_m+1})\cdot A,$$
            where $b_{j_2}$ is replaced by $b_{j_1+1}$ if $m=1$. The cycle $\barb$ is obtained from $B$ by removing the element $b_{j_1}$ and $\bara$ from $A$ by adding $b_{j_m+1}$ between $a_{i_m+1}$ and $a_{i_m}$.
            \item\label{caso4m} If $a_{i_m+1}<b_{j_m+1}$, let $$\barb=(b_{j_m+1},a_{i_m+1})\cdot(b_{j_1},b_{j_2})\cdot B\hspace{5mm} \text{and} \hspace{5mm}\bara=A,$$
            where $b_{j_2}$ is replaced by $b_{j_1+1}$ if $m=1$. Here, $\barb$ is equal to $B$ but with $b_{j_1}$ removed and $a_{i_m+1}$ added between $b_{j_m}$ and $b_{j_m+1}$.
        \end{enumerate}       
\end{itemize}
Now, we drop hypothesis \eqref{distanza}. Let $l_1,\ldots,l_t\in\{1,\ldots,m-1\}$ be all the indexes such that $$i_{l_r+1}-i_{l_r}>1 \hspace{5mm}\text{or}\hspace{5mm} j_{l_r+1}-j_{l_r}>1,$$ for all $r\in\{1,2,\ldots,t\},$ 
then we consider the blocks $$\{c_1,\ldots,c_{l_1}\},\{c_{l_1+1},\ldots,c_{l_2}\},\ldots,\{c_{l_t+1},\ldots,c_m\}$$ separately. In other words, we consider the least fine partition of $\{c_1,c_2,\ldots,c_m\}$ such that in each block hypothesis \eqref{distanza} holds. The idea is to look at the position of $b_{j_{l_{r-1}+1}}$ in the one-line notation of $u$ with respect to that of $b_{j_{l_{r-1}+1}-1}$, for $r\in\{1,\ldots,t\}$ one at a time, and to blend the results obtained for $t=0$, each time depending on whether $a_{i_{l_r}+1}$ is greater than $b_{j_{l_r}+1}$.
\newline
More formally, we start by considering the first block. We apply to the cycles $A$ and $B$ the transformations corresponding to the case in which we are depending on the relative position of $b_{j_1}$ and $b_{j_1-1}$ in the one-line notation of $u$ and on whether $a_{i_{l_1}+1}$ is greater than $b_{j_{l_1}+1}$. We let $A_1$ and $B_1$ be the resulting cycles. Then, we do the same for the second block: we apply to $A_1$ and $B_1$ the transformations corresponding to the case in which we are depending on the relative position of $b_{j_{l_1+1}}$ and $b_{j_{l_1+1}-1}$ in the one-line notation of $u$ and on whether $a_{i_{l_2}+1}$ is greater than $b_{j_{l_2}+1}$. We let $A_2$ and $B_2$ be the resulting cycles. We proceed in the same way for all the $t+1$ blocks. Eventually, we let $\barb=B_{t+1}$ and $\bara=A_{t+1}$.
\newline
Notice that this is possible because each block does not interact with the others, meaning that the involutions corresponding to each block commute with those corresponding to any other block. In fact, according to the construction above, when considering the $r\text{-th}$ block $\{c_{l_{r-1}+1},\ldots,c_{l_r}\}$, the only elements involved in the transformations in the four cases altogether are $$\mathfrak{A}_r=\{a_{i_{l_{r-1}+1}-1},a_{i_{l_{r-1}+1}},a_{i_{l_r}},b_{j_{l_{r}}+1}\} \hspace{5mm}\text{ and }\hspace{5mm} \mathfrak{B}_r=\{b_{j_{l_{r-1}+1}}, b_{j_{l_{r-1}+1}+1},a_{i_{l_{r}}+1},b_{j_{l_r}+1}\}$$
for the cycle $A$ and $B$ respectively. So, if both $i_{l_{r}+1}-i_{l_r}>1$ and $j_{l_{r}+1}-j_{l_r}>1$, we have no problem when we apply the transformations regarding the $r+1\text{-th}$ block $\{c_{l_r+1},\ldots,c_{l_{r+1}}\}$ since $\mathfrak{A}_r\cap\mathfrak{A}_{r+1}=\emptyset$ and $\mathfrak{B}_r\cap\mathfrak{B}_{r+1}=\emptyset$. 
We check the other possibilities. 
\begin{itemize}[label=$-$]
    \item If $i_{l_{r}+1}-i_{l_r}>1$ and $j_{l_{r}+1}-j_{l_r}=1$, then $b_{j_{l_r}+1}=b_{j_{l_r+1}}\in\mathfrak{B}_r\cap\mathfrak{B}_{r+1}$. Though, the one-line notation of $u$ must be of the form 
    $$u=[\hspace{0.7mm}\cdots \hspace{0.7mm}\substack{a_{ i_{l_r}-1}\\  \parallel \\b_{j_{l_r}-1}}\hspace{0.7mm}\cdots \hspace{0.7mm}\hspace{0.7mm}\substack{a_{ i_{l_r}}\\  \parallel \\b_{j_{l_r}}}\hspace{0.7mm}\cdots\hspace{0.7mm}a_{i_{l_r}+1}\hspace{0.7mm}\cdots \hspace{0.7mm} a_{i_{l_r}+2}\hspace{0.7mm} \cdots \cdots\hspace{0.7mm}a_{i_{l_{r}+1}-1}\hspace{0.7mm}\cdots\hspace{0.7mm}\substack{a_{ i_{l_{r}+1}}\\  \parallel \\b_{j_{l_{r}+1}}}\hspace{0.7mm} \cdots\hspace{0.7mm} \substack{a_{ i_{l_{r}+1}+1}\\  \parallel \\b_{j_{l_{r}+1}+1}}\hspace{0.7mm}\cdots\hspace{0.7mm}]$$
    and hence $u^{-1}(b_{j_{l_r+1}-1})=u^{-1}(b_{j_{l_r}})<u^{-1}(b_{j_{l_r+1}})$, which means that the second block can belong to cases \ref{caso1m} and \ref{caso2m} only. Moreover, notice that $b_{j_{l_r}+1}=c_{l_r+1}=a_{i_{l_r+1}}>a_{i_{l_r}+1}$, which means that the first block can belong to cases \ref{caso2m} and \ref{caso4m} only.
    \begin{itemize}
        \item[2.] Assume that, in the first block, we are in case \ref{caso2m}, that is $u^{-1}(b_{j_{l_{r-1}+1}-1})<u^{-1}(b_{j_{l_{r-1}+1}})$ and $a_{i_{l_r}+1}<b_{j_{l_r}+1}$. Then, we let $B_r=(b_{j_{l_r}+1},a_{i_{l_r}+1})\cdot B$ and $A_r=(a_{i_{l_{r-1}+1}},a_{i_{l_{r-1}+1}-1})\cdot A$.
        \begin{itemize}
            \item[1.] If the second block belongs to case \ref{caso1m}, then we let $B_{r+1}=B_r=(b_{j_{l_r}+1},a_{i_{l_r}+1})\cdot B$ and $A_2=(a_{i_{l_{r+1}}},b_{j_{l_{r+1}}+1})\cdot (a_{i_{l_r+1}},a_{i_{l_r+1}-1})\cdot A_r=(a_{i_{l_{r+1}}},b_{j_{l_{r+1}}+1})\cdot (a_{i_{l_r+1}},a_{i_{l_r+1}-1})\cdot(a_{i_{l_{r-1}+1}},a_{i_{l_{r-1}+1}-1})\cdot A $.
            \item[2.] If the second block belongs to case \ref{caso2m}, then we let $B_{r+1}=(b_{j_{l_{r+1}}+1},a_{i_{l_{r+1}}+1})\cdot B_r=(b_{j_{l_{r+1}}+1},a_{i_{l_{r+1}}+1})\cdot(b_{j_{l_r}+1},a_{i_{l_r}+1})\cdot B$ and $A_{r+1}=(a_{i_{l_{r}+1}},a_{i_{l_{r}+1}-1})\cdot A_r=(a_{i_{l_{r}+1}},a_{i_{l_{r}+1}-1})\cdot(a_{i_{l_{r-1}+1}},a_{i_{l_{r-1}+1}-1})\cdot A $.
        \end{itemize}
        \item[4.] Assume that, in the first block, we are in case \ref{caso4m}, that is $u^{-1}(b_{j_{l_{r-1}+1}-1})>u^{-1}(b_{j_{l_{r-1}+1}})$ and $a_{i_{l_r}+1}<b_{j_{l_r}+1}$. Then, we let $B_r=(b_{j_{l_r}+1},a_{i_{l_r}+1})\cdot(b_{j_{l_{r-1}+1}},b_{j_{l_{r-1}+1}+1})\cdot B$ and $A_r=A$.
        \begin{itemize}
           \item[1.] If the second block belongs to case \ref{caso1m}, then we let $B_{r+1}=B_r=(b_{j_{l_r}+1},a_{i_{l_r}+1})\cdot(b_{j_{l_{r-1}+1}},b_{j_{l_{r-1}+1}+1})\cdot B$ and $A_{r+1}=(a_{i_{l_{r+1}}},b_{j_{l_{r+1}}+1})\cdot (a_{i_{l_r+1}},a_{i_{l_r+1}-1})\cdot A_r=(a_{i_{l_{r+1}}},b_{j_{l_{r+1}}+1})\cdot (a_{i_{l_r+1}},a_{i_{l_r+1}-1})\cdot A $.
            \item[2.] If the second block belongs to case \ref{caso2m}, then we let $B_{r+1}=(b_{j_{l_{r+1}}+1},a_{i_{l_{r+1}}+1})\cdot B_r=(b_{j_{l_{r+1}}+1},a_{i_{l_{r+1}}+1})\cdot(b_{j_{l_r}+1},a_{i_{l_r}+1})\cdot(b_{j_{l_{r-1}+1}},b_{j_{l_{r-1}+1}+1})\cdot B$ and $A_{r+1}=(a_{i_{l_{r}+1}},a_{i_{l_{r}+1}-1})\cdot A_r=(a_{i_{l_{r}+1}},a_{i_{l_{r}+1}-1})\cdot A $.
        \end{itemize}
    \end{itemize}
    \item If $i_{l_{r}+1}-i_{l_r}=1$ and $j_{l_{r}+1}-j_{l_r}>1$, then $a_{i_{l_r+1}-1}=a_{i_{l_r}}\in\mathfrak{A}_r\cap\mathfrak{A}_{r+1}$. Though, the one-line notation of $u$ must be of the form $$u=[\hspace{0.3mm}\cdots \hspace{0.3mm}\substack{a_{ i_{l_r}-1}\\  \parallel \\b_{j_{l_r}-1}}\hspace{0.3mm}\cdots \hspace{0.3mm}\hspace{0.3mm}\substack{a_{ i_{l_r}}\\  \parallel \\b_{j_{l_r}}}\hspace{0.3mm}\cdots\hspace{0.3mm}\substack{a_{ i_{l_{r}+1}}\\  \parallel \\b_{j_{l_{r}+1}}}\hspace{0.3mm} \cdots\hspace{0.3mm}b_{j_{l_r}+1}\hspace{0.3mm}\cdots \hspace{0.3mm} b_{j_{l_r}+2}\hspace{0.3mm} \cdots \cdots\hspace{0.3mm}b_{j_{l_{r}+1}-1}\hspace{0.3mm}\cdots\hspace{0.3mm} \substack{a_{ i_{l_{r}+1}+1}\\  \parallel \\b_{j_{l_{r}+1}+1}}\hspace{0.3mm}\cdots\hspace{0.3mm} \substack{a_{ i_{l_{r}+1}+2}\\  \parallel \\b_{j_{l_{r}+1}+2}}\hspace{0.3mm}\cdots\hspace{0.3mm}]$$
    and hence $u^{-1}(b_{j_{l_r+1}-1})>u^{-1}(b_{j_{l_r+1}})$, which means that the second block can belong to cases \ref{caso3m} and \ref{caso4m} only. Moreover, notice that $b_{j_{l_r}+1}<b_{j_{l_r+1}}=a_{i_{l_r+1}}=a_{i_{l_r}+1}$, which means that the first block can belong to cases \ref{caso1m} and \ref{caso3m} only.
    \begin{itemize}
        \item[1.] Assume that, in the first block, we are in case \ref{caso1m}, that is $u^{-1}(b_{j_{l_{r-1}+1}-1})<u^{-1}(b_{j_{l_{r-1}+1}})$ and $a_{i_{l_r}+1}>b_{j_{l_r}+1}$. Then, we let $B_r= B$ and $A_r=(a_{i_{l_r}},b_{j_{l_r}+1})\cdot(a_{i_{l_{r-1}+1}},a_{i_{l_{r-1}+1}-1})\cdot A$.
        \begin{itemize}
            \item[3.] If the second block belongs to case \ref{caso3m}, then we let $B_{r+1}=(b_{j_{l_r+1}},b_{j_{l_r+1}+1})\cdot B_r=(b_{j_{l_r+1}},b_{j_{l_r+1}+1})\cdot B$ and $A_{r+1}= (a_{i_{l_{r+1}}},b_{j_{l_{r+1}}+1})\cdot A_r= (a_{i_{l_{r+1}}},b_{j_{l_{r+1}}+1})\cdot (a_{i_{l_r}},b_{j_{l_r}+1})\cdot(a_{i_{l_{r-1}+1}},a_{i_{l_{r-1}+1}-1})\cdot A $.
            \item[4.] If the second block belongs to case \ref{caso4m}, then we let $B_{r+1}=(b_{j_{l_{r+1}}+1},a_{i_{l_{r+1}}+1})\cdot(b_{j_{l_r+1}},b_{j_{l_r+1}+1})\cdot B_r=(b_{j_{l_{r+1}}+1},a_{i_{l_{r+1}}+1})\cdot(b_{j_{l_r+1}},b_{j_{l_r+1}+1})\cdot B$ and $A_{r+1}=  A_r= (a_{i_{l_r}},b_{j_{l_r}+1})\cdot(a_{i_{l_{r-1}+1}},a_{i_{l_{r-1}+1}-1})\cdot A $.
        \end{itemize}
        \item[3.] Assume that, in the first block, we are in case \ref{caso3m}, that is $u^{-1}(b_{j_{l_{r-1}+1}-1})>u^{-1}(b_{j_{l_{r-1}+1}})$ and $a_{i_{l_r}+1}>b_{j_{l_r}+1}$. Then, we let $B_r= (b_{j_{l_{r-1}+1}},b_{j_{l_{r-1}+1}+1})\cdot B$ and $A_1=(a_{i_{l_r}},b_{j_{l_r}+1})\cdot A$.
        \begin{itemize}
            \item[3.] If the second block belongs to case \ref{caso3m}, then we let $B_{r+1}=(b_{j_{l_r+1}},b_{j_{l_r+1}+1})\cdot B_r=(b_{j_{l_r+1}},b_{j_{l_r+1}+1})\cdot (b_{j_{l_{r-1}+1}},b_{j_{l_{r-1}+1}+1})\cdot B$ and $A_{r+1}= (a_{i_{l_{r+1}}},b_{j_{l_{r+1}}+1})\cdot A_r= (a_{i_{l_{r+1}}},b_{j_{l_{r+1}}+1})\cdot (a_{i_{l_r}},b_{j_{l_r}+1})\cdot A $.
            \item[4.] If the second block belongs to case \ref{caso4m}, then we let $B_{r+1}=(b_{j_{l_{r+1}}+1},a_{i_{l_{r+1}}+1})\cdot(b_{j_{l_r+1}},b_{j_{l_r+1}+1})\cdot B_r=(b_{j_{l_{r+1}}+1},a_{i_{l_{r+1}}+1})\cdot(b_{j_{l_r+1}},b_{j_{l_r+1}+1})\cdot(b_{j_{l_{r-1}+1}},b_{j_{l_{r-1}+1}+1})\cdot B$ and $A_{r+1}=  A_r= (a_{i_{l_r}},b_{j_{l_r}+1})\cdot A $.
        \end{itemize}
    \end{itemize}
    
\end{itemize}
This allows us to conclude that, in any case, the transformations corresponding to some block do not overlap with those of another block, and thus the recipe for computing $\phi(A,B)$ is well-defined.

In this way, we completed the definition of the map $\phi$. 

\underline{Definition of $\psi$:}\\ We now show that it is a bijection by constructing its inverse $\psi:\bar{\mathcal{D}}\to \mathcal{D}$. Let $(\barb,\bara)\in \mathcal{\bar D}$. As for $\phi$, we define $\psi(\barb,\bara)=(\widetilde A,\widetilde B)$ case-by-case.  Let $\barb=(1,\bar b_1,\bar b_2,\ldots,\bar b_{\bar h})$, $\bara=(n,\bar a_{\bar k},\ldots, \bar a_2,\bar a_1)$, and denote $\bar b_{\bar h+1}:=1$ and $\bar a_{\bar k +1}:=n$. For each case, the following conditions hold:
\begin{enumerate}[label=(\roman*')]
    \item\label{c1} $\widetilde A$ and $\widetilde B$ are cycles satisfying \ref{shzw1}, \ref{shzw2},  \ref{bshzw1}, \ref{bshzw2};
\item\label{c2} $\widetilde B \cdot \widetilde A=\bara \cdot \barb$;
\item\label{c3} $|\widetilde A_s|+|\widetilde B_s|=|\bar{B}_s|+|\bar{A}_s|$.
\end{enumerate}

If $|\bara_s\cap\barb_s|=0$, we let $\widetilde A=\bara$ and $\widetilde B=\barb$. Since $\bara$ and $\barb$ commute, conditions \ref{c1}, \ref{c2}, and \ref{c3} are immediate.

If $|\bara_s\cap \bar B_s|=m$, for some $m>0$, we let $\{\bar c_1,\bar c_2,\ldots,\bar c_m\}$ be the elements of $\bar A_s\cap \bar B_s$ indexed in non-decreasing order. Then  $\bar c_l=\bar a_{i_l}=\bar b_{j_l}$ for some $i_l\in \{1,\ldots,\bar k, \bar k+1\}$, $j_l\in\{1,\ldots,\bar h\}$ such that $i_1<i_2<\ldots<i_m$ and $j_1<j_2<\ldots<j_m$. 
\newline
As for the definition of $\phi$, we start by assuming that \begin{equation*}
\label{distanzainv}\tag{$\ast\ast$}
i_{r+1}-i_r=1 \hspace{3mm}\text{ and }\hspace{3mm}
j_{r+1}-j_r= 1 \hspace{3mm} \text{ for all } \hspace{3mm} r\in\{1,\ldots,m-1\}.
\end{equation*}
Notice that, in the one-line notation of $u$, the elements $\bar b_1,\bar b_2,\ldots,\bar b_{\bar h}$ must be in this order. The elements $\bar a_1,\bar a_2,\ldots,\bar a_{\bar k}$ must appear in order as well, except for $\bar a_{i_m}$. In fact, it may happen that $\bar a_{i_m}$ is to the right of $\bar a_{i_m+1}$, since $\bar a_{i_m}\in \barb_s$ and therefore it will be moved to the left when cycle $\barb$ is applied to $u$. The same cannot happen for $\bar a_{i_l}$, for $l\in\{1,\ldots,m-1\}$, even though $\bar a_{i_l}\in \barb_s$, because $\bar a_{i_l}=\bar b_{j_l}$.
\newline
Hence, we distinguish two main cases depending on whether $\bar a_{i_m}$ lies to the left or to the right of $\bar a_{i_m+1}$ in the one-line notation of $u$.  Overall, there are four cases, as each main case splits into two sub-cases depending on whether $\bar b_{j_1-1}$ is greater than $\bar a_{i_1-1}$. If $i_1=1$, then $\bar a_{i_1-1}$ is undefined; in this instance, we adopt the convention to be in the case $\bar b_{j_1-1}>\bar a_{i_1-1}$ and then let $\bar a_{i_1-1}:=n$. If $j_1=1$, we just let $\bar b_0:=1$.
\begin{itemize}     
        \item
        If $u^{-1}(\bar a_{i_m})<u^{-1}(\bar a_{i_m+1})$, then the one-line notation of $u$ is of the form
        $$u=[\hspace{0.7mm}\cdots \hspace{0.7mm} \bar b_{j_1-1}\hspace{0.7mm}\cdots \hspace{0.7mm}\substack{a_{ i_1}\\  \parallel \\b_{j_1}}\hspace{0.7mm}\cdots  \cdots \hspace{0.7mm}\substack{a_{ i_{m-1}}\\  \parallel \\b_{j_{m-1}}}\hspace{0.7mm}\cdots\hspace{0.7mm}\substack{\bar a_{i_m}\\  \parallel \\\bar b_{j_m}}\hspace{0.7mm}\cdots \hspace{0.7mm}\bar a_{i_m+1}\hspace{0.7mm} \cdots \hspace{0.7mm} ].$$
        \begin{enumerate}[label= \arabic*'.]
            \item\label{invcaso1} If $\bar b_{j_1-1}<\bar a_{i_1-1}$, let $\widetilde A=\bara$ and $\widetilde B =(\bar b_{j_1},\bar a_{i_1-1})\cdot(\bar b_{j_m},\bar b_{j_m+1})\cdot \barb$. 
 \item\label{invcaso2} If $\bar b_{j_1-1}>\bar a_{i_1-1}$, let $\widetilde A=(\bar a_{i_1-1},\bar b_{j_1-1})\cdot \bara$ and $\widetilde B=(\bar b_{j_m},\bar b_{j_m+1})$. 
 \end{enumerate}
\item
    If $u^{-1}(\bar a_{i_m})>u^{-1}(\bar a_{i_m+1})$, then the one-line notation of $u$ will be of the form
    $$u=[\hspace{0.7mm}\cdots \hspace{0.7mm} \bar b_{j_1-1}\hspace{0.7mm}\cdots \hspace{0.7mm}\substack{a_{ i_1}\\  \parallel \\b_{j_1}}\hspace{0.7mm}\cdots\cdots \hspace{0.7mm}\substack{a_{ i_{m-2}}\\  \parallel \\b_{j_{m-2}}}\hspace{0.7mm}\cdots  \hspace{0.7mm}\substack{a_{ i_{m-1}}\\  \parallel \\b_{j_{m-1}}}\hspace{0.7mm}\cdots\hspace{0.7mm}\bar a_{i_m+1}\hspace{0.7mm} \cdots \hspace{0.7mm}\substack{\bar a_{i_m}\\  \parallel \\\bar b_{j_m}}\hspace{0.7mm}\cdots \hspace{0.7mm} ].$$
        \begin{enumerate}[resume,label= \arabic*'.]
            \item\label{invcaso3} If $\bar b_{j_1-1}<\bar a_{i_1-1}$, let $\widetilde A=(\bar a_{i_m},\bar a_{i_{m-1}})\cdot \bara$ and $\widetilde B=(\bar b_{j_1},\bar a_{i_1-1})\cdot \barb$, where $\bar a_{i_{m-1}}$ is replaced by $\bar a_{i_m-1}$ if $m=1$.
            \item\label{invcaso4} If $\bar b_{j_1-1}>\bar a_{i_1-1}$, let $\widetilde A=(\bar a_{i_1-1},\bar b_{j_1-1})\cdot (\bar a_{i_m},\bar a_{i_{m-1}})\cdot \bara$ and $\widetilde B=\barb$, where $\bar a_{i_{m-1}}$ is replaced by $\bar a_{i_m-1}$ if $m=1$.
        \end{enumerate}  
\end{itemize}
Notice that this definition is symmetric to that of $\phi$ (see Table \ref{tabella} below). More formally, we have that $$\psi(\barb,\bara)=\mathcal C_{w_0}\circ\phi_{[w_0uw_0,w_0vw_0]}\circ \mathcal{C}_{w_0}(\barb,\bara),$$
where $\mathcal{C}_{w_0}$ is the conjugation by the longest element $w_0$ of $\mathfrak{S}_n$, and $\phi_{[w_0uw_0,w_0vw_0]}$ is the map $\phi$ but associated with the interval $[w_0uw_0,w_0vw_0]$.

Hence, the generalization to the general case in which hypothesis \eqref{distanzainv} does not hold can be done in exactly the same way we did for $\phi$ when we dropped hypothesis \eqref{distanza}.
\begin{table}[h]
    \centering
    \renewcommand{\arraystretch}{1.8} 
    \begin{tabular}{c !{\vrule width 1.2pt} c|c|c|c|c|c|c|c|c|c}
        $\phi$ & $(A,B)$ & $(\barb,\bara)$ & $<$ & $b_{j_1}$ & $b_{j_1-1}$ & $a_{i_m+1}$ & $b_{j_m+1}$ & $a_{i_m}$ & $a_{i_1-1}$ & $b_{j_2}$ \\ 
        \Xhline{1.2pt} 
        $\psi$ &$(\barb,\bara)$&$(\widetilde A,\widetilde B)$& $>$ & $\bar a_{i_m}$ & $\bar a_{i_m+1}$ & $\bar b_{j_1-1}$ & $\bar a_{i_1-1}$ & $\bar b_{j_1}$ & $\bar b_{j_m+1}$ & $\bar a_{i_{m-1}}$ \\  
    \end{tabular}
    \caption{Correspondence between the definitions of $\phi$ and $\psi$}
    \label{tabella}
\end{table}

\underline{Proof that $\psi=\phi^{-1}$:} \\Now, let $(A,B)\in\mathcal  {D}$. We claim that $\psi(\phi(A,B))=(A,B)$. If $A_s\cap B_s=\emptyset$, the claim is trivially verified. If $|A_s\cap B_s|\geq 1$, we prove the validity of the claim under the hypothesis \eqref{distanza}, the other cases being just a generalization. We distinguish the four main cases.
\begin{itemize}[label=-]
    \item If $(A,B)$ lies in case \ref{caso1m}, then $(\barb,\bara)$ lies in case \ref{invcaso4} In fact, $\bar A_s\cap\bar B_s=\{c_2,\ldots,c_m,b_{j_m+1}\}$, $\bar{a}_{i_1-1}=a_{i_1-1}<a_{i_1}=b_{j_1}=\bar b_{j_1-1}$, and $u^{-1}(\bar{a}_{i_m})=u^{-1}(b_{j_m+1})>u^{-1}(a_{i_m+1})=u^{-1}(\bar a_{i_m+1})$. We obtain $\widetilde A=(a_{i_1-1},b_{j_1-1})\cdot (b_{j_m+1},a_{i_{m}})\cdot \bara =(a_{i_1-1},a_{i_1})\cdot (b_{j_m+1},a_{i_{m}})\cdot (a_{i_m},b_{j_m+1})\cdot(a_{i_1},a_{i_{1}-1})\cdot A=A$ and $\widetilde B=\barb=B$.
    \item If $(A,B)$ lies in case \ref{caso2m}, then $(\barb,\bara)$ lies in case \ref{invcaso2} In fact, $\bar A_s\cap\bar B_s=\{c_2,\ldots,c_m,a_{i_m+1}\}$, $\bar{a}_{i_1-1}=a_{i_1-1}<a_{i_1}=b_{j_1}=\bar b_{j_1-1}$, and $u^{-1}(\bar{a}_{i_m+1})=u^{-1}(a_{i_m+2})>u^{-1}(b_{j_m+1})=u^{-1}(\bar a_{i_m})$. We obtain $\widetilde A=(a_{i_1-1},a_{i_1})\cdot \bar A=(a_{i_1-1},a_{i_1})\cdot (a_{i_1},a_{i_1-1})\cdot A$ and $\widetilde B=(a_{i_m+1},b_{j_m+1})\cdot\bar B
    =(a_{i_m+1},b_{j_m+1})\cdot(b_{j_m+1},a_{i_m+1})\cdot B$.
    \item If $(A,B)$ lies in case \ref{caso3m}, then $(\barb,\bara)$ lies in case \ref{invcaso3} In fact, $\bar A_s\cap\bar B_s=\{c_2,\ldots,c_m,b_{j_m+1}\}$, $\bar{a}_{i_1-1}=a_{i_1}=b_{j_1}>b_{j_1-1}=\bar b_{j_1-1}$, and $u^{-1}(\bar{a}_{i_m+1})=u^{-1}(a_{i_m+1})<u^{-1}(b_{j_m+1})=u^{-1}(\bar a_{i_m})$. We obtain $\widetilde A=(b_{j_m+1},a_{i_m})\cdot\bara=(b_{j_m+1},a_{i_m})\cdot(a_{i_m},b_{j_m+1})\cdot A$ and $\widetilde B=(b_{j_2},b_{j_1})\cdot\bar B=(b_{j_2},b_{j_1})\cdot(b_{j_1},b_{j_2})\cdot B$.
    \item If $(A,B)$ lies in case \ref{caso4m}, then $(\barb,\bara)$ lies in case \ref{invcaso1} In fact, $\bar A_s\cap\bar B_s=\{c_2,\ldots,c_m,a_{i_m+1}\}$, $\bar{a}_{i_1-1}=a_{i_1}=b_{j_1}>b_{j_1-1}=\bar b_{j_1-1}$, and $u^{-1}(\bar{a}_{i_m+1})=u^{-1}(a_{i_m+2})>u^{-1}(a_{i_m+1})=u^{-1}(\bar a_{i_m})$. We obtain $\widetilde A=\bara=A$ and $\widetilde B=(b_j,a_{i+1})\cdot\barb\cdot(b_j,a_{i+1})=(a_{i_m+1},b_{j_m+1})\cdot(b_{j_2},b_{j_1})\cdot (b_{j_m+1},a_{i_m+1})\cdot (b_{j_1},b_{j_2}) \cdot B$.
\end{itemize}
Hence, $\psi=\phi^{-1}$.
\end{proof}

\begin{thm}\label{thm11}
Let $W$ be a Coxeter group of type $A_{n-1}$, and let $u,v \in W$. Then the following hold:
\begin{enumerate}
\item 
\label{thm11_1}
$DS(z_n,z_1)=DS(z_1,z_n)$;
\item 
\label{thm11_2}
$DS(z^n, z^1)=DS(z^1, z^n)$;
\item 
\label{thm11_3}
$DS(z_n,z^n)=DS(z^n,z_n)$;
\item 
\label{thm11_4}
$DS(z^1,z_1)=DS(z_1,z^1)$;
\end{enumerate}
\end{thm}
\begin{proof}
\begin{enumerate}[leftmargin=*]
    \item It follows from the existence of a bijection between  $\mathcal{D}$ and $\mathcal{\bar D}$, which is the content of  Lemma~\ref{lemma biezione1}. 
\item It is obtained from (\ref{thm11_1}) by considering the poset isomorphism $x\mapsto x^{-1}$ from $[u,v]$ to $[u^{-1},v^{-1}]$. In fact, $b$ is a $(z^n,z^1)$-double shortcut for $[u,v]$ if and only if $b^{-1}$ is a $(\hat z_n,\hat z_1)$-double shortcut for $[u^{-1},v^{-1}]$, where $\hat z_n=\min([u^{-1},v^{-1}]\cap W_{S\setminus\{s_{n-1}\}}v^{-1})$ and $\hat z_1=\min([u^{-1},v^{-1}]\cap W_{S\setminus\{s_{1}\}}v^{-1})$.
\item It is obtained from (\ref{thm11_1}), but replacing 1 with $\alpha:=v(n)$ in the definition of the sets $\mathcal{D}$ and $\mathcal{\bar D}$, meaning that we should consider the sets \[
\mathcal{D}' := \left\{ (A,B) : 
\parbox{7.8cm}{
    $A$ and $B$ are cycles $(n, a_k, a_{k-1},\ldots, a_1)$ and $(\alpha, b_1, b_2,\ldots, b_h)$ 
    satisfying \ref{2shzw1}, \ref{2shzw2}, \ref{2bshzw1}, \ref{2bshzw2}
} \right\},
\]
where \begin{enumerate}[label=(\arabic*)]
\item \label{2shzw1}
$1\leq a_1<a_2<\ldots <a_k<n$,
\item \label{2shzw2} $v^{-1}(n)=u^{-1}(a_1)<u^{-1}(a_2)<\ldots <u^{-1}(a_k)<u^{-1}(n)$,
\item\label{2bshzw1} $\alpha<b_1<\ldots<b_{h-1}<b_h < n$,
\item\label{2bshzw2} $u^{-1}\cdot A^{-1}(\alpha)<u^{-1}\cdot A^{-1}(b_1)<\ldots<u^{-1}\cdot A^{-1}(b_{h-1})<u^{-1}\cdot A^{-1}(b_h)=n$, 
\end{enumerate}
and
\[
\mathcal{\bar D}' := \left\{ (\barb,\bara) : 
\parbox{8.1cm}{
     $\bar{B}$ and $\bar{A}$ are cycles  $(\alpha,\bar b_1,\bar b_{2},\ldots,\bar b_{\bar h})$ and $(n,\bar a_{\bar k}, \bar a_{\bar k-1}, \ldots,\bar a_1)$ satisfying \ref{2mshzw1}, \ref{2mshzw2},  \ref{2mbshzw1}, \ref{2mbshzw2}}
     \right\},
\]
where \begin{enumerate}[label=(\alph*)]
\item \label{2mshzw1}
$\alpha<\bar b_1<\ldots<\bar b_{\bar{h}-1}<\bar b_{\bar h}\leq n$,
\item \label{2mshzw2}
$u^{-1}(\alpha)<u^{-1}(\bar b_1)<\ldots<u^{-1}(\bar b_{{\bar h}-1})<u^{-1}(\bar b_{\bar h})=v^{-1}(1)$. 

\item \label{2mbshzw1}
$1\leq \bar a_1<\bar a_2<\ldots <\bar a_{\bar k}<n$, $\bar a_i\neq \alpha$,
\item \label{2mbshzw2}
$v^{-1}(n)=u^{-1}\cdot \barb^{-1}(\bar a_1)<u^{-1}\cdot \barb^{-1}(\bar a_2)<\ldots <u^{-1}\cdot \barb^{-1}(\bar a_{\bar k})<u^{-1}\cdot \barb^{-1}(n)$, 
\end{enumerate}
and then proceed exactly as before, again replacing 1 with $\alpha$ inside the proof of the lemma.
                
\item 
It is obtained from (\ref{thm11_3}) by considering the poset isomorphism $x\mapsto w_0xw_0$ from $[u,v]$ to $[w_0uw_0,w_0vw_0]$. In fact, $b$ is a $(z_1,z^1)$-double shortcut for $[u,v]$ if and only if $w_0bw_0$ is a $(\hat z_n,\hat z^n)$-double shortcut for $[w_0uw_0,w_0vw_0]$, where $\hat z_n=\min([w_0uw_0,w_0vw_0]\cap W_{S\setminus\{s_{n-1}\}}w_0vw_0)$ and $\hat z^n=\min([w_0uw_0,w_0vw_0]\cap w_0vw_0W_{S\setminus\{s_{n-1}\}})$.
\end{enumerate}
\end{proof}

The following result implies that Conjecture~\ref{EMtrivial} holds for those intervals having no hypercube decompositions apart from the standard ones.
\begin{cor}\label{cor}
Let $u,v\in W$. 
Then standard hypercube decompositions of $[u,v]$ belong to the same equivalence class.
\end{cor}
\begin{proof}
    This is an easy consequence of Theorem~\ref{thm11}. By (\ref{thm11_1}), we have that $DS(z_n,z_1)=DS(z_1,z_n)$; hence $z_n\sim z_1$. By (\ref{thm11_3}), we have that $DS(z_n,z^n)=DS(z^n,z_n)$; hence $z_n\sim z^n$. By (\ref{thm11_4}), we have that $DS(z^1,z_1)=DS(z_1,z^1)$; hence $z_1\sim z^1$. The claim then follows by transitivity.
\end{proof}

\end{document}